\documentclass[11pt,oneside]{article}
\usepackage[utf8]{inputenc}
\usepackage[english]{babel}
\usepackage{amsmath}
\usepackage{amsfonts}
\usepackage{amssymb}
\usepackage{amsthm}
\usepackage{tikz-cd} 
\usepackage[left=3cm,right=3cm,top=3cm,bottom=3cm]{geometry}
\usepackage[title]{appendix}
\usepackage[]{algorithm2e}

\usepackage[normalem]{ulem}
\useunder{\uline}{\ul}{}
\theoremstyle{plain}
\newtheorem{teo}{}[section]
\newtheorem{prop}[teo]{Proposition}

\newtheorem{lem}[teo]{Lemma}
\newtheorem{thm}[teo]{Theorem}
\newtheorem{df}[teo]{Definition}
\theoremstyle{definition}
\newtheorem{ex}[teo]{Example}
\newtheorem{rem}[teo]{Remark}

\newcommand\blfootnote[1]{%
  \begingroup
  \renewcommand\thefootnote{}\footnote{#1}%
  \addtocounter{footnote}{-1}%
  \endgroup
}

\title{A combinatorial description of shape theory}
\author{Pedro J. Chocano, Manuel A. Morón and Francisco R. Ruiz del Portal}
\date{}

\begin{document}
\maketitle

\begin{abstract}
We give a combinatorial description of shape theory using finite topological $T_0$-spaces (finite partially ordered sets). This description may lead to a sort of computational shape theory. Then we introduce the notion of core for inverse sequences of finite spaces and prove some properties.
\end{abstract}

\section{Introduction}\label{sec:introduccion}
\blfootnote{2020  Mathematics  Subject  Classification:   06A11,06A07, 54C56, 55P55.}
\blfootnote{Keywords: Posets, finite topological space, algebraic topology and shape theory.}
\blfootnote{This research is partially supported by Grant PGC2018-098321-B-100 from Ministerio de Ciencia, Innovación y Universidades (Spain).}

The theory of dynamical systems is a very active area of researching. This is due to its wide variety of applications to other areas of science such as physics, biology or engineering \cite{strogatz2015nonlinear,raduska2015applications,katok1995nonlinear}. Because of the need of finding computational methods to study dynamical systems, the theory of finite topological spaces \cite{may1966finite,barmak2011algebraic} has recently grown up in this direction. Classical topological methods in dynamical systems, the Conley index \cite{conley1978isolated,szymczak1995theconley,mrozek1990leray} or the Lefschetz fixed point theorem \cite{lefschetz1937on}, have been adapted to this framework \cite{lipinski2019conley,barmak2011lefschetz,chocano2020coincidence}. This combinatorial approach has lead to a sort of persistence algorithms that can be used to analyze data collected from dynamical systems \cite{mrozek2019persistent,mrozek2020persistence,mrozek2021persistenceConleyMorse}.  The classical Morse theory \cite{milnor1963Morse} has also been adapted for finite spaces in \cite{minian2012some,fernandez2020morse} and for simplicial complexes in \cite{forman1998combinatorial}. This reformulation of the Morse theory for simplicial complexes has led to interesting applications of computational aspects in a fruitful manner, see for example \cite{curry2015discrete} or \cite{harker2013discrete}. Therefore, it is natural to think that adapting classical notions we could have computational applications to the general study of dynamical systems, among others. 

Moreover, shape theory appears as a generalization of homotopy theory. It provides a weaker classification of compact metric spaces than the homotopy theory. Originally, it was developed to study global properties of compact metric spaces that do not necessarily have good local properties. Since some dynamical objects such as attractors do not behave well locally, shape theory has several applications there (for a recent account of this treatment we refer the reader to \cite{gabites2008dynamical} and its references). For example, the Conley index can also be defined using this theory \cite{robbin1988dynamical}. There are different approaches to shape theory: the original approach of K. Borsuk \cite{borsuk1971shape}, the categorical approach of S. Mardeši\'c and J. Segal \cite{mardevsic1982shape}, the categorical approach of J.M. Cordier and T. Porter \cite{cordier1989shape} or the intrinsic description using multivalued maps of J.M.R. Sanjurjo \cite{sanjurjo1992intrinsic}. In general, the main idea is to think about a topological space $X$ as an inverse system of ``easier'' spaces approximating $X$, where the morphisms are given in terms of these systems.   

Recently, methods to reconstruct compact metric spaces using finite topological spaces have been developed. To every compact metric space, it was associated in \cite{moron2008connectedness} an inverse sequence of finite spaces and continuous maps between them. Later in \cite{clader2009inverse}, E. Clader constructed for every simplicial complex $X$ an inverse sequence of finite spaces satisfying that its inverse limit contains a homeomorphic copy of $X$ which is a strong deformation retract. A generalization of this result to compact metric spaces was given in \cite{mondejar2015hyperspaces} and in \cite{chocano2021computational}, where computational aspects are studied and a method to reconstruct algebraic invariants is proposed. These methods are motivated somehow by the recent theory of topological data analysis \cite{edelsbrunner2008topologicalData,carlsson2009topology}, which has grown up very fast the last years (e.g. see \cite{otter2017roadmap} and the references given there). It is also worth noting that there is a more general method to reconstruct locally compact, paracompact and Hausdorff spaces using Alexandroff spaces (partially ordered sets) \cite{bilski2017inverse}, but it is not suitable from a computational viewpoint because it uses in general non-finite spaces.

Therefore, it is reasonable to ask whether it is possible to describe shape theory in terms of these inverse sequences. In that way, we may have a sort of computational or combinatorial shape theory. There are other previous results that point out some of the relations between finiteness and shape theory, see \cite{moron2001finite,sanjurjo1997density} or \cite{chocano2021shape}. In \cite{sanjurjo1997density} an intrinsic description of shape theory is given using sequences of continuous maps defined on open dense subsets of the compact metric spaces and having finite images. In \cite{moron2001finite} it is given an intrinsic representation of Čech homology of compacta in terms of inverse limits of discrete approximations. In \cite{chocano2021shape} it is constructed a category that classifies compacta by their shape and finite topological spaces by their weak homotopy type.

The organization of this paper is as follows. In Section \ref{sec:preliminares} we introduce the reformulation of the intrinsic description of shape theory \cite{sanjurjo1992intrinsic} that was given in \cite{moron2008homotopical}, we also recall the method to reconstruct algebraic invariants given in \cite{chocano2021computational} and basic results in the theory of finite spaces. In Section \ref{sec:descripcion} we describe our combinatorial approach of shape theory, define the notion of core for inverse sequences and get some shape invariants. In Section \ref{sec:resultadoPrincipal} we show the main result: the shape category of compact metric spaces is isomorphic to a category  whose morphisms are described in terms of finite topological spaces.  

\section{Preliminaries}\label{sec:preliminares}

Given a compact metric space $(X,d)$, where $d$ denotes the metric, we consider the so-called hyperspace of $X$, $2^X=\{C\subseteq X| C$ is non-empty and closed$\}$. Let $B(U)=\{C\in 2^X|C\subset U\}$ for every open set $U\subseteq X$. Then the family $B=\{B(U)| U\subseteq X$ is open$\}$ is a base for the upper semifinite topology on $2^X$. Moreover, $X$ can be embedded in $2^X$ because $X$ is a $T_2$-space, it suffices to consider $\varphi:X\rightarrow 2^X$ given by $\varphi(x)=\{ x\}$. For a complete exposition about this topic see \cite{nadler1978hyperspaces}. Write $U_\epsilon=\{C\in 2^X| \text{diam}(C)<\epsilon \}$, where $\text{diam}(C)$ denotes the diameter of $C$ and $\epsilon$ is a positive real value. We recall some properties of hyperspaces with the upper semifinite topology studied in \cite{moron2008homotopical} and \cite{mondejar2015hyperspaces}.

\begin{prop}\label{prop:baseEntornosAbiertos} Let $(X,d)$ be a compact metric space. Then the family $\mathcal{U}=\{U_\epsilon \}_{\epsilon>0}$ is a base of open neighborhoods of the embedding of $X$ in $2^X$.
\end{prop}
\begin{lem}\label{lem:elevacionHiperespacio} Let $Z$ and $T$ be compact metric spaces and let $h:Z\rightarrow 2^T$ be a continuous map. Then $h^*:2^Z\rightarrow 2^T$ given by $h^*(C)=\bigcup_{c\in C} f(c)$ is well-defined and continuous.
\end{lem}
%
The following results use the construction $U_\epsilon$ described above for two compact metric spaces $X$ and $Y$, so we will denote them by $U_\epsilon(X)$ and $U_\epsilon(Y)$ respectively.
\begin{thm}\label{thm:extensionHomotopiaHyperspaces} Let $X$ and $Y$ be compact metric spaces. If $H:X\times [0,1]\rightarrow 2^Y$ and $h:2^X\rightarrow 2^Y$ are continuous maps such that $H(x,0)=h_{|X}(\{ x\})$, then there exists a map $\overline{H}:2^X\times [0,1]\rightarrow 2^Y$ satisfying the following properties:
\begin{enumerate}
\item $\overline{H}(C,0)=h(C)$ for all $C\in 2^X$.
\item $\overline{H}_{|X\times [0,1]}=H$.
\item $\overline{H}$ is continuous.
\item If $H(x,t)\in U_\epsilon(Y)$ for all $(x,t)\in X\times [0,1]$, then there exists $\gamma>0$ such that $\overline{H}(U_{\gamma}(X)\times [0,1])\subset U_{\epsilon}(Y)$.

\end{enumerate}
\end{thm}
Now, we recall the description of the shape theory that was given in \cite{moron2008homotopical} using the results obtained in \cite{sanjurjo1992intrinsic}.
\begin{df}\label{def:approximativeMap} Given two compact metric spaces $X$ and $Y$, a sequence of continuous functions $\overline{f}=\{f_k:X\rightarrow 2^Y \}_{k\in \mathbb{N}}$ is said to be an approximative map from $X$ to $Y$ if for every neighborhood $U$ of the canonical copy $Y$ in $2^Y$ there exists $k_0\in \mathbb{N}$ such that $f_k$ is homotopic to $f_{k+1}$ in $U$ for all $k\geq k_0$.
\end{df}
\begin{df}\label{def:approximativeMapHomotopia} Given two approximative maps $\overline{f}$ and $\overline{g}$ from $X$ to $Y$, $\overline{f}$ is homotopic to $\overline{g}$ if for each open neighborhood $U$  of the canonical copy $Y$ in $2^Y $ there exists $n_0$ such that $f_n$ is homotopic to $g_n$ in $U$ for every $n\geq n_0$.
\end{df}
\begin{thm}\label{thm:descripcionShapeMoronSanjurjo} The set of all homotopy classes of approximative maps from $X$ to $Y$ is in bijective correspondence with the set of shape morphisms from $X$ to $Y$.
\end{thm}
We adapt the notion of composition of multivalued maps in \cite{sanjurjo1992intrinsic} to approximative maps.  Let $[f]:X\rightarrow Y$ and $[g]:Y\rightarrow Z$ be two classes of approximative maps, $\overline{f}=\{f_k:X\rightarrow 2^Y \}_{k\in \mathbb{N}}\in [f]$,  $\overline{g}=\{g_k:X\rightarrow 2^Y \}_{k\in \mathbb{N}}\in [g]$. Let $\{\epsilon_n\}_{n\in \mathbb{N}}$ be a strictly decreasing sequence of positive real values such that $g_{n}$ is homotopic to $g_{n_0}$ in $U_{\epsilon_{n_0}}(Z)$ for every $n\geq n_0$ and let $\{\nu_n\}_{n\in \mathbb{N}}$ be a strictly decreasing sequence of positive real values such that $\text{diam}(g_{n}(K))<\epsilon_n$ whenever $K$ is a set in $Y$ with $\text{diam}(K)<\nu_n$. Now consider a sequence of indices $k_1<k_2<...<k_n<...$ such that $f_k$ is homotopic to $f_{k_n}$ in $U_{\nu_n}(Y)$ for every $k\geq k_n$. We define $[g]\circ [f]=[\{g_n\circ f_{k_n}:X\rightarrow 2^Z \}_{n\in \mathbb{N}}]$. From \cite[Proof of Theorem 3]{sanjurjo1992intrinsic} and \cite[Proposition 18]{moron2008homotopical} it may be concluded that the composition of approximative maps is an approximative map and it does not depend on the representatives of the homotopy classes $[f]$ and $[g]$. 

We get a category whose objects are compact metric spaces and whose morphisms are the homotopy classes of approximative maps. The identity morphism is given by the homotopy class of $\{i_n:X\rightarrow 2^X\}_{n\in \mathbb{N}}$, where $i_n(x)=\{x\}$ for every $n\in \mathbb{N}$ and $x\in X$. Let $HN$ denote this category. 
\begin{thm}\label{thm:approximativeCategory} $HN$ is isomorphic to the shape category of compact metric spaces.
\end{thm}
\begin{rem} It can be said that this approach substitutes the Hilbert cube by hyperspaces in the original approach given by K. Borsuk \cite{borsuk1971shape}.
\end{rem}

We recall briefly the method obtained in \cite{chocano2021computational} to reconstruct algebraic invariants of $X$. We say that a finite set $A\subset X$ is an $\epsilon$-approximation of $X$ if for every $x\in X$ there exists $a\in A$ such that $d(x,a)<\epsilon$, where $\epsilon$ is a positive real value. Given an $\epsilon$-approximation $A$ of $X$, we define $\mathcal{U}_{4\epsilon}(A)=\{C\subset A|\text{diam}(C)<4\epsilon \}$. Since $A$ is a finite set, we get that $\mathcal{U}_{4\epsilon}(A)=\{C\subset A|\text{diam}(C)<4\epsilon \}$ is a finite partially ordered set with the following relation: $C\leq D$ if and only if $C\subseteq D$. Thus, $\mathcal{U}_{4\epsilon}(A)=\{C\subset A|\text{diam}(C)<4\epsilon \}$ is a finite topological $T_0$-space because the category of finite topological $T_0$-spaces and finite partially ordered sets (posets) are isomorphic \cite{alexandroff1937diskrete}. From now on, we denote the category of finite posets by $Poset$ and we assume that every finite spaces is $T_0$.

Let $T$ be a finite space and $x\in T$. We denote by $U_x$ the intersection of every open set containing $x$, which is again open. We say that $x\leq y$ if and only if $U_x\subseteq U_y$. With this relation we have that $X$ is a finite poset. Suppose $(X,\leq)$ is a finite poset. Then the family of lower sets of $\leq$ is a $T_0$ topology on $X$. A lower set $S$ of $X$ is a set satisfying that if $x\leq y\in S$, then $x\in S$. See \cite{may1966finite} or \cite{barmak2011algebraic} for a complete introduction to the theory of finite topological spaces. In what follows, we treat finite posets and finite topological $T_0$-spaces as the same object. Note that a map $f:X\rightarrow Y$ between finite topological spaces is continuous if and only if it is order-preserving. 

The Hasse diagram of a finite poset $X$ is a directed graph: the vertices are the points of $X$ and there is an edge $(x,y)$ if and only if $x<y$ and there is no $z\in X$ satisfying $x<z<y$. We assume an upward orientation in subsequent graphs.

An important concept in the theory of finite spaces is the notion of weak homotopy equivalence, we recall it. A map $f:X\rightarrow Y$ is a weak homotopy equivalence if the induced maps $f_*:\pi_i(X,x)\rightarrow \pi_i(Y,f(x))$ are isomorphisms for all $x\in X$ and all positive integer number $i$, where the map $f_*$ induces a bijection in dimension $0$. Notice that every weak homotopy equivalence induces isomorphisms on singular homology groups by a well-known theorem of J.H.C. Whitehead.
 
There is a functor $\mathcal{K}:Poset\rightarrow SimpComplex$, where $SimpComplex$ denotes the category of simplicial complexes whose morphisms are simplicial maps. As usual, $|K|$ stands for the geometric realization of a simplicial complex $K$. For a finite poset $X$, $\mathcal{K}(X)$ denotes the order complex of $X$. There is also a functor $\mathcal{X}:SimpComplex \rightarrow Poset$. For a simplicial complex $L$, $\mathcal{X}(L)$ denotes the face poset of $L$. The finite barycentric subdivision of a finite poset $X$ is defined by $\mathcal{X}(\mathcal{K}(X))$. 
\begin{thm}\label{thm:McCord} For each finite topological space $X$ there exists a weak homotopy equivalence $f:|\mathcal{K}(X)|\rightarrow X$. For each finite simplicial complex $K$ there exists a weak homotopy equivalence $f:|K|\rightarrow \mathcal{X}(K)$.
\end{thm}
\begin{thm}\label{thm:comutatividadDiagramaMcCord} Let $g:X\rightarrow Y$ be a continuous function between finite spaces and let $f_X:|\mathcal{K}(X)|\rightarrow X$ and $f_Y:|\mathcal{K}(Y)|\rightarrow Y$ denote the weak homotopy equivalences of Theorem \ref{thm:McCord}. Then $g\circ f_X=f_Y \circ |\mathcal{K}(f)|$.
\end{thm}
Theorems \ref{thm:McCord} and \ref{thm:comutatividadDiagramaMcCord} can be found in \cite{mccord1966singular}. For a fuller treatment about finite spaces see \cite{may1966finite} or \cite{barmak2011algebraic}.

Let $(X,d)$ be a compact metric space and let $\{\epsilon_n\}_{n\in \mathbb{N}}$ be a sequence of positive real values satisfying that $\epsilon_{n+1}<\frac{\epsilon_n}{2}$. For every $n\in \mathbb{N}$ we consider an $\epsilon_n$-approximation $A_n$. The map $q_{n,n+1}:\mathcal{U}_{4\epsilon_{n+1}}(A_{n+1})\rightarrow \mathcal{U}_{4\epsilon_n}(A_n)$ given by $q_{n,n+1}(C)=\bigcup_{c\in C}\mathcal{B}(c,\epsilon_n)\cap A_n$ is continuous (see \cite{chocano2021computational}), where $\mathcal{B}(x,\epsilon)$ denotes the open ball of radius $\epsilon$ and center $x$. We say that the inverse sequence $(\mathcal{U}_{4\epsilon_n}(A_n),q_{n,n+1})$ is a \textbf{finite approximation} of $X$.

Let $H_*$ denote the homological functor, where we consider the singular homology.

\begin{prop}[\cite{chocano2021computational}]\label{prop:reconstruccionHomologia} Given a compact metric space $(X,d)$ and a finite approximation $(\mathcal{U}_{4\epsilon_n}(A_n),q_{n,n+1})$ of it. The inverse limit of $(H_l((\mathcal{U}_{4\epsilon_n}(A_n)),H_l(q_{n,n+1}))$ is isomorphic to the $l$-dimensional Čech homology group of $X$.
\end{prop}

Note that if $X$ is a CW-complex, then the singular homology groups of $X$ coincide with the Čech homology groups of $X$. As for prerequisites, the reader is expected to be familiar with the notion of pro-category. Concretely, we will use the categories pro-$HTop$ and pro-$Top$, where $Top$ is the topological category and $HPol$ is the homotopy category of topological spaces. For more details about this topic, inverse systems and inverse sequences we refer the reader to \cite{mardevsic1982shape}.

\begin{rem}\label{rem:mainConstruction} There is also a similar construction that uses other bonding maps. Let $(X,d)$ be a compact metric space and let $\{\delta_n\}_{n\in \mathbb{N}}$ be a sequence of positive real values satisfying that $\epsilon_{n+1}<\frac{\epsilon_n}{2}$. For every $n\in \mathbb{N}$ we consider an $\delta_n$-approximation $A_n$. We define $p_{n,n+1}:\mathcal{U}_{2\delta_{n+1}}(A_{n+1})\rightarrow \mathcal{U}_{2\delta_n}(A_n)$ by $p_{n,n+1}(C)=\bigcup_{x\in C}\{ a\in A_n|d(a,x)=d(A,x)\}$. We get that $\mathcal{U}_{2\delta_n}(A_n)$ is a finite poset with the subset relation and $(\mathcal{U}_{2\delta_n}(A_n),p_{n,n+1})$ is an inverse sequence. This inverse sequence is similar to the Main Construction introduced in \cite{moron2008connectedness} and is isomorphic to every finite approximation $(\mathcal{U}_{4\epsilon_n}(A_n),q_{n,n+1})$ of $X$ in pro-$HTop$ (see\cite{chocano2021computational}). 
\end{rem}

\section{Combinatorial description of shape theory and shape invariants}\label{sec:descripcion}
The idea of the combinatorial description of the shape theory is to use finite approximations of compact metric spaces to define morphisms between them. Given a compact metric space $(X,d)$, we construct from finite samples of $X$ an inverse sequence of finite topological spaces, i.e., the finite approximation of $X$ introduced in Section \ref{sec:preliminares}. This construction is not unique since it depends on the finite samples of $X$ and the values of $\epsilon_n$. However, in \cite{chocano2021computational} it is shown that given two finite approximations of $X$, they are isomorphic in pro-$HTop$. Hence, we can say that a finite approximation $(\mathcal{U}_{4\epsilon_n}(A_n),q_{n,n+1})$ of $X$ is unique in a suitable category and we denote it by $T(X)$.

We define the category $\mathbb{E}$ as follows. The objects of $\mathbb{E}$ are compact metric spaces. Given two compact metric spaces $X$ and $Y$, $\mathbb{E}(X,Y)=\{(f_n,f):T(X)\rightarrow T(Y)|$ $(f_n,f)$ is a morphism in pro-$HTop$ where $T(X)$ is a finite approximation of $X$ and $T(Y)$ is a finite approximation of $Y\}$. It is easy to check that $\mathbb{E}$ is a category. In Figure \ref{fig:descripcionShape} we have depicted a schematic representation. We now state the main result and develop some notions and shape invariants.

\begin{figure}[h]
\centering
\includegraphics[scale=0.9]{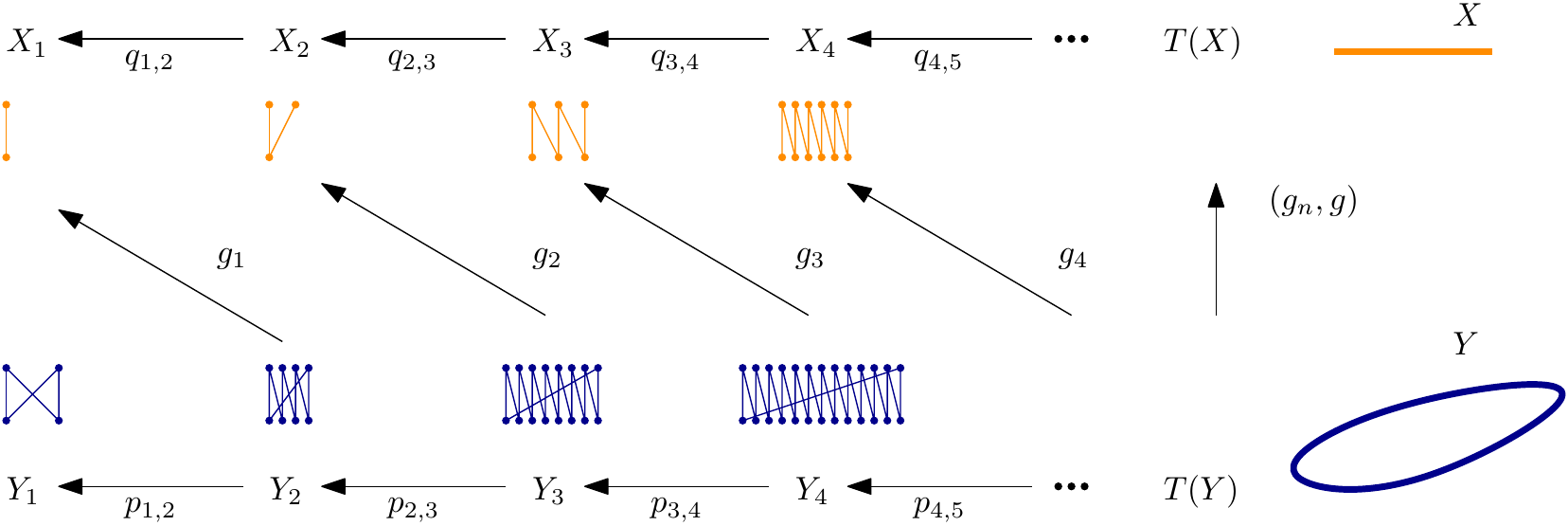}
\caption{Schematic description of the combinatorial approach of shape theory.}\label{fig:descripcionShape}
\end{figure}

\begin{thm}\label{thm:shapeTheorem} The category $\mathbb{E}$ is isomorphic to the shape category of compact metric spaces.
\end{thm}
For finite topological spaces there exists the notion of core. We recall it for completeness. 
\begin{df} Let $X$ be a finite topological space and $x\in X$. It is said that $x$ is a down (up) beat point if $U_x\setminus \{x \}$ ($F_x\setminus \{x \}$) has a maximum  (minimum).
\end{df}
\begin{prop}\label{prop:core} If $X$ is a finite topological space and $x\in X$ is a beat point, then $X\setminus \{ x\}$ is a strong deformation retract of $X$.
\end{prop}
A finite space is a \textbf{minimal finite space} if it does not have beat points. The \textbf{core} of a finite space $X$ is the resulting space after removing one by one beat points until there are no more. We generalize this notion to inverse sequences of finite spaces.
\begin{df} Let $(X_n,t_{n,n+1})$ be an inverse sequence of finite topological spaces. Suppose $C_n$ is the core of $X_n$ and $r_n:X_n\rightarrow C_n$ is a retraction satisfying that $r_n\circ i_n=\textnormal{id}_{C_n}$ and $i_n\circ r_n$ is homotopic to $\textnormal{id}_{X_n}$, where $i_n:C_n\rightarrow X_n$ denotes the inclusion. We say that the core of $(X_n,t_{n,n+1})$ is $(C_n,r_{n}\circ t_{n,n+1}\circ i_{n+1})$ and write $C(X_n,t_{n,n+1})=(C_n,r_{n}\circ t_{n,n+1}\circ i_{n+1})$.
\end{df}
By Proposition \ref{prop:core} we get that every finite space and its core are isomorphic in $HTop$, we prove the analogue result for inverse sequences of finite spaces and their cores.
\begin{thm}\label{thm:CoreFAsigualFAS} Let $(X_{n},q_{n,n+1})$ be an inverse sequence of finite topological spaces. Then $(X_{n},q_{n,n+1})$ is isomorphic to $C(X_{n},q_{n,n+1})$ in pro-$HTop$.
\end{thm}
\begin{proof}
We fix notation, $C(X_{n},q_{n,n+1})=(C_n,h_{n,n+1})$ where $h_{n,n+1}=r_{n}\circ t_{n,n+1}\circ i_{n+1}$. There is a natural morphism in pro-$HTop$ between $(C_n,h_{n,n+1})$ and $(X_{n},q_{n,n+1})$ induced by the inclusions, that is, $i:\mathbb{N}\rightarrow \mathbb{N}$ is the identity map and $i_n:C_n\rightarrow X_n$ is the inclusion. It is trivial to check that $(i_n,i)$ is a well-defined morphism since the following diagram is commutative up to homotopy.
\[
  \begin{tikzcd}[row sep=large,column sep=huge]
C_n \arrow{d}{i_{n}}  &  C_{n+1}\arrow{d}{i_{n+1}} \arrow[l,"h_{n,n+1}"'] \\
X_n  & X_{n+1} \arrow{l}{q_{n,n+1}}
  \end{tikzcd}
\]

We now construct a sequence $\{g_n:X_{n+1}\rightarrow C_n\}_{n\in \mathbb{N}}$ of continuous maps making the following diagram commutative up to homotopy.
\[
  \begin{tikzcd}[row sep=large,column sep=huge]
C_n \arrow{d}{i_n}  &  C_{n+1}\arrow{d}{i_{n+1}} \arrow[l,"h_{n,n+1}"'] \\
X_n  & X_{n+1} \arrow{l}{q_{n,n+1}} \arrow{lu}{g_n}
  \end{tikzcd}
\]
For every $n\in \mathbb{N}$ we define $g_n:X_{n+1}\rightarrow C_n$ by $g_{n}=r_n\circ q_{n,n+1}$. By construction, we have $g_n\circ i_{n+1}=r_n\circ q_{n,n+1}\circ i_{n+1}$ and $h_{n,n+1}=r_{n}\circ q_{n,n+1}\circ i_{n+1}$, which yields the commutativity of the first triangle. We also have that $i_{n}\circ g_n=i_n\circ r_n\circ q_{n,n+1}$. Therefore, $i_{n}\circ g_n$ is homotopic to $q_{n,n+1}$ and we get the commutativity up to homotopy of the  second triangle. By Morita's lemma (see \cite{morita1974hurewicz} or \cite[Chapter 2, Theorem 5]{mardevsic1982shape}), we get the desired result.  
\end{proof}

\begin{rem}\label{rem:strongdeformationretractcomoelcore}Let $(X_n,q_{n,n+1})$ be an inverse sequence of finite spaces. Suppose $L_n$ is a strong deformation retract of $X_n$ for every $n\in \mathbb{N}$. Following the same arguments used before, we can obtain an inverse sequence where the terms are given by $L_n$. Repeating the proof of Theorem \ref{thm:CoreFAsigualFAS}, it can be deduced that this new inverse sequence is isomorphic to $(X_n,q_{n,n+1})$ in pro-$HTop$. 
\end{rem}

\begin{ex} We consider $X=\{A,B \}$ and declare that $A<B$. Let $X^n$ denote the $n$-th finite barycentric subdivision of $X$. The finite barycentric subdivision of $X$ can be seen as the poset given by the chains of $X$ where the partial order is given by the subset relation. We have a natural inverse sequence given by $(X^n,h_{n,n+1})$, where $h_{n,n+1}:X^{n+1}\rightarrow X^n$ is given by $h(x_1<...<x_m)=x_m$ (see \cite{barmak2011algebraic} or \cite{may1966finite} for more details). It is easily seen that the core of $X^n$ is a space with one point for every $n\in \mathbb{N}$. This implies that $(X^n,h_{n,n+1})$ is isomorphic to  $(*,\text{id})$. In Figure \ref{fig:coreInversesystem} we have an schematic representation.
\begin{figure}[h]
\centering
\includegraphics[scale=0.9]{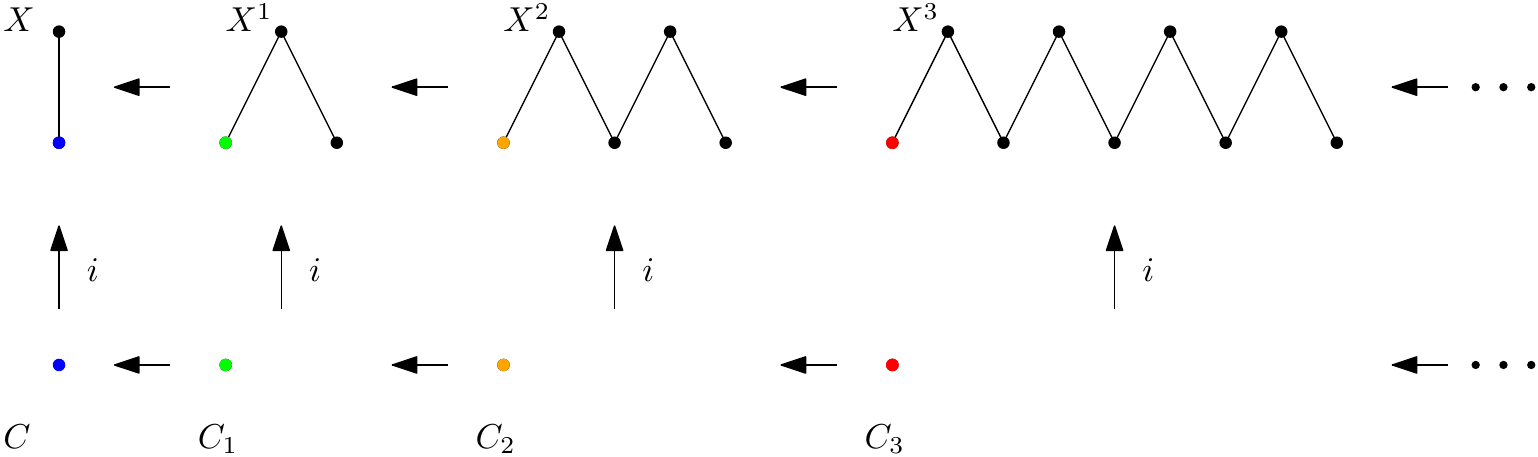}
\caption{Schematic illustration of the inverse sequence $(X^n,h_{n,n+1})$ and its core.}\label{fig:coreInversesystem}

\end{figure} 
\end{ex}
The core of an inverse sequence may be used to show that two compact metric spaces have the same shape, we give an example of this and also illustrate the way we can use different constructions of finite spaces (see Remark \ref{rem:mainConstruction}).

\begin{ex}\label{ex:curve} We consider the computational model of the topologist's sine curve $S$, that is, 
\begin{align*}
S=a_\infty\cup (\cup_{n\geq 1} \overline{b}_n) \cup(\cup_{n\geq 1} \underline{b}_n)\cup (\cup_{n\geq 0} a_n).
\end{align*}
where
\begin{align*}
\overline{b}_n & =(\frac{1}{2^{2n-1}},\frac{1}{2})-(\frac{1}{2^{2n-2}},\frac{1}{2}) \quad n\geq 1 , \\
\underline{b}_n & =(\frac{1}{2^{2n}},\frac{1}{2})-(\frac{1}{2^{2n-1}},0) \quad n\geq 1 ,\\
a_n & = (\frac{1}{2^n},\frac{1}{2})-(\frac{1}{2^n},0) \quad n\geq 0 ,\\
a_\infty &=(0,\frac{1}{2})-(0,0),
\end{align*}
$n\in \mathbb{N}$ and $(a,b)-(c,d)$ denotes the segment joining the point $(a,b)$ with $(c,d)$ (see Figure \ref{fig:computationaltopologistsinecurve}). The metric of $S$ is the one inherited as a subspace of $\mathbb{R}^2$. We get the Main Construction for $S$ and study at the same time its core.

\begin{figure}[h]
 
  \centering
    \includegraphics[scale=0.6]{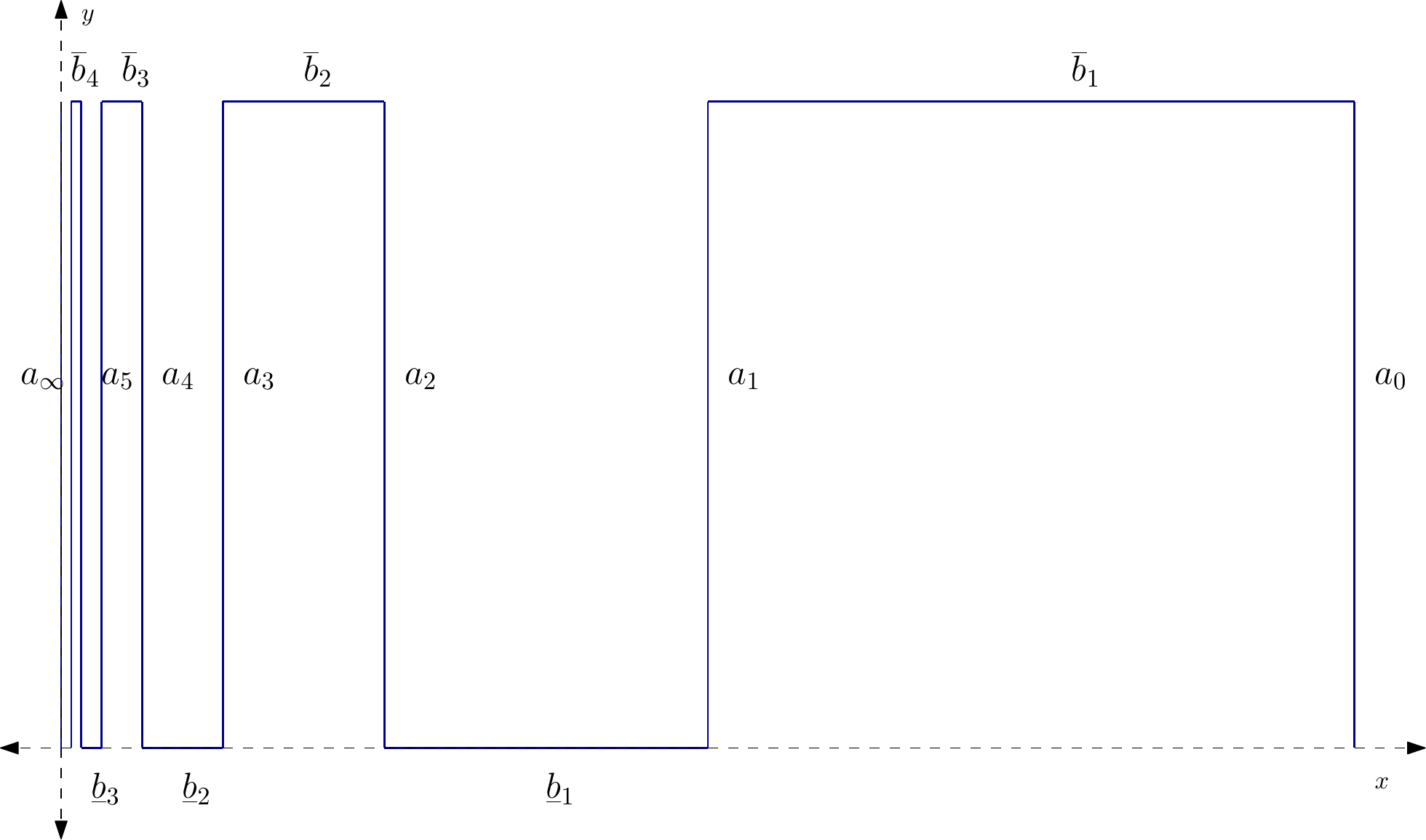}
     \caption{The computational topologist's sine curve.}\label{fig:computationaltopologistsinecurve}
\end{figure}

\underline{Step 1.} The diameter of $S$ is $\frac{\sqrt{2}}{2}$, so we can consider $\epsilon_1=\sqrt{5}$, $A_1=\{(0,\frac{1}{4}) \}$ and $\mathcal{U}_{2\epsilon_{1}}(A_1)=A_1$.  The core of $\mathcal{U}_{2\epsilon_{1}}(A_1)$ is also $A_1$.

\underline{Step 2.} We consider $\epsilon_2=\frac{\sqrt{2}}{2^3}<\frac{\epsilon_1}{2}$, the grid $G_2=\{(\frac{l}{2^{3-1}},\frac{k}{2^{3-1}})\in \mathbb{R}^2| l,k\in \mathbb{Z}\}$ and the intersection of $G_2$ with $S$. There are two points in $a_3$ that are at distance $\epsilon_2$ to $G_2\cap S$, which are
$b_1=(\frac{1}{2^3},\frac{1}{2^3}) $ and $ b_2=(\frac{1}{2^3},\frac{1}{2^3}+\frac{1}{2^2}) $. If we add $(0,\frac{1}{2^3})$ and $(0,\frac{1}{2^3}+\frac{1}{2^2})$ to $G_2\cap T$, then we get an $\epsilon_2$-approximation
$$A_2=G_2\cap S \cup \{ (0,\frac{1}{2^3}),(0,\frac{1}{2^3}+\frac{1}{2^2}) \}.$$

\begin{figure}[!htb]
   \begin{minipage}{0.42\textwidth}
   \centering
     \includegraphics[scale=0.4]{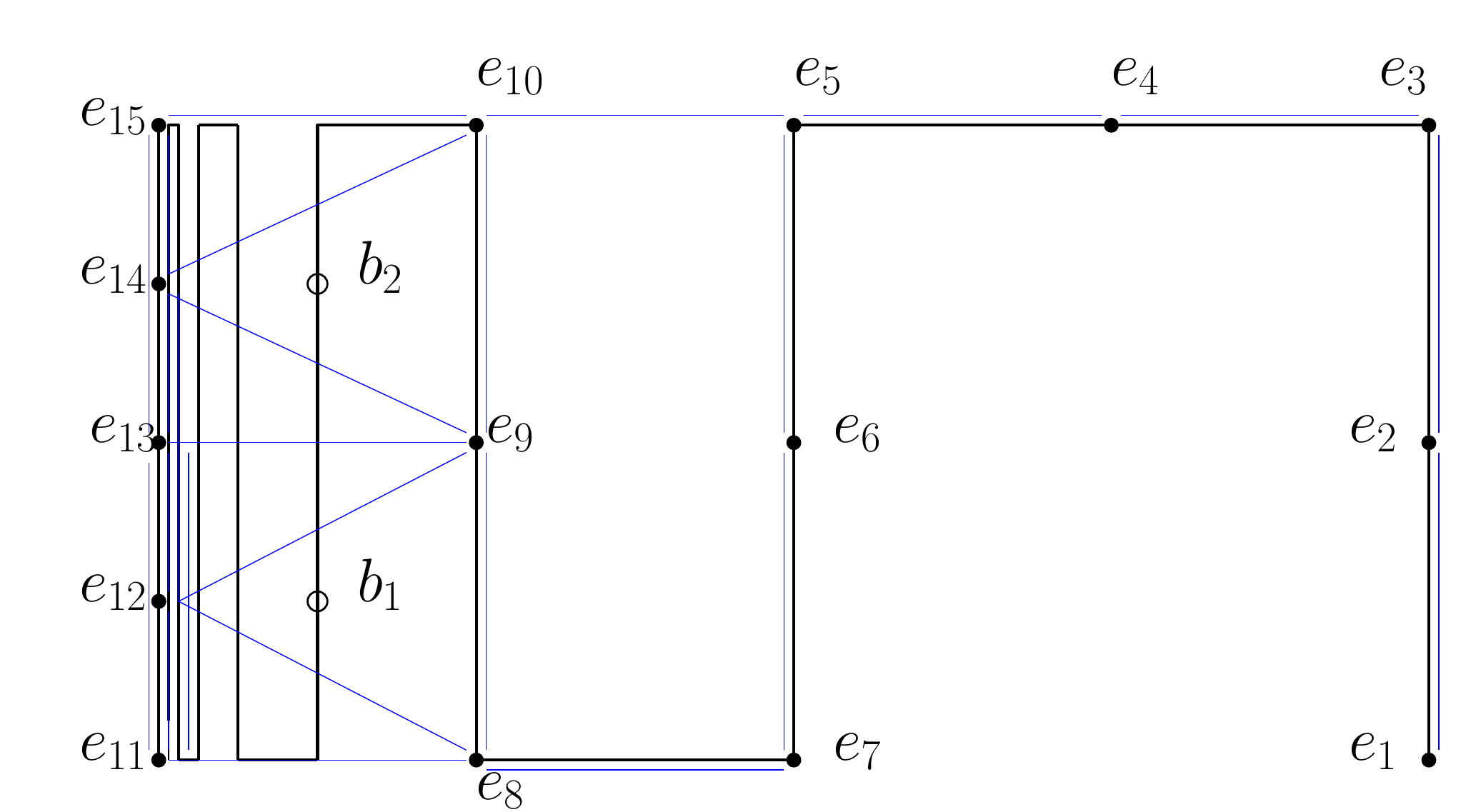}
\caption{$\epsilon_2$-approximation $A_2$.}\label{fig:A_2Aproximacion}
   \end{minipage}\qquad \qquad \qquad
   \begin{minipage}{0.42\textwidth}
   \centering
     \includegraphics[scale=0.4]{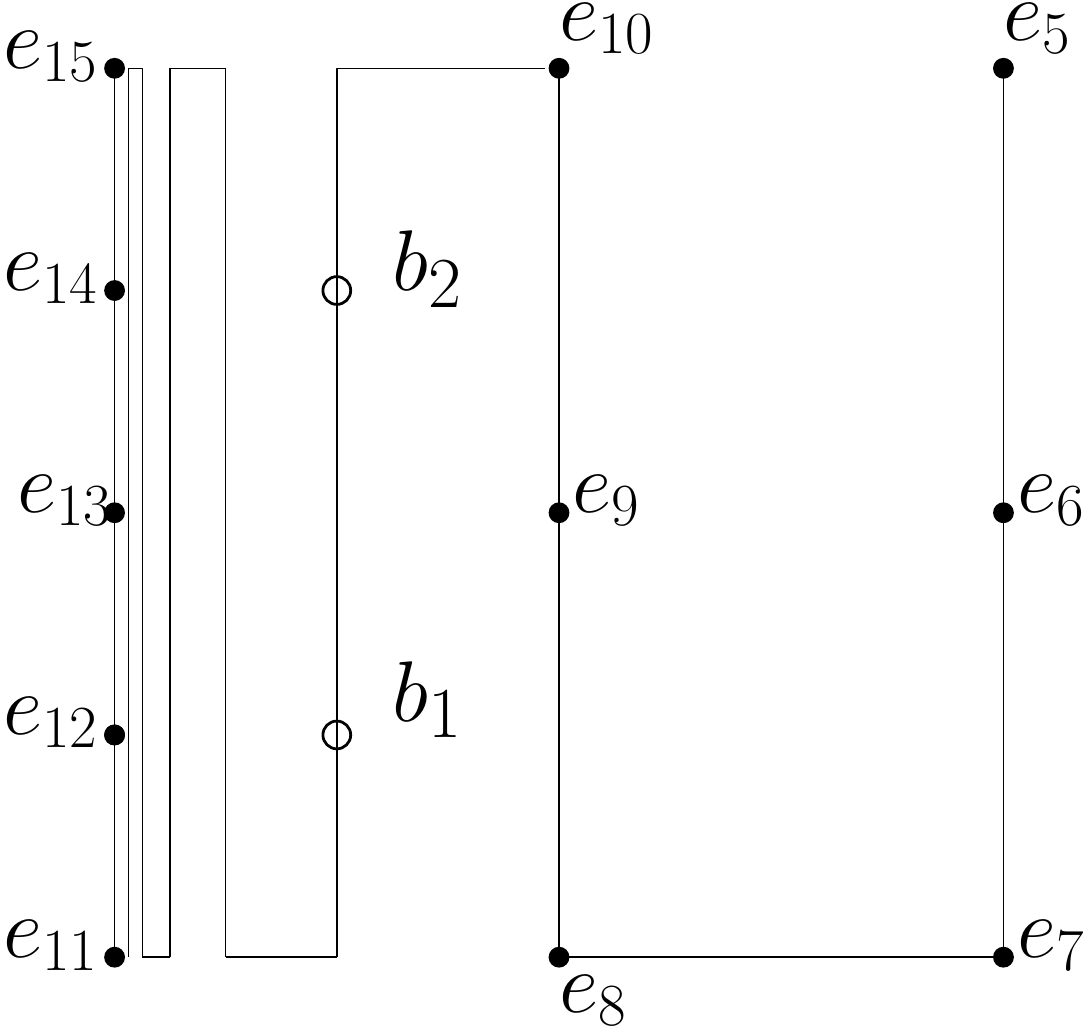}
\caption{The set of points $A_2'$ in $S$.}\label{fig:A_2AproximacionCore}
   \end{minipage}
\end{figure}

We write $B_2=\{ x\in A_2| x\in \overline{b}_1\setminus{a_1} $ or  $x\in a_0 \}$, that is, the points of $A_2$ that lie to the right of $a_1$. It is easy to observe that $\mathcal{U}_{2\epsilon_2}(A'_2)$ is a strong deformation retract of $\mathcal{U}_{2\epsilon_2}(A_2)$, where $A_2'=A_2\setminus B_2$. The last assertion is an immediate consequence of the construction we made of $A_2$ and the value that we have chosen for $\epsilon_2$. Suppose $C\in \mathcal{U}_{2\epsilon_2}(A_2)$ contains points of the approximation that lie in $B_2$, which means that $C$ is of the form $C=\{e_k,e_{k+1}\}$ or $C=\{e_k\}$. Then $\{e_1\}$ is an up beat point because $F_{\{e_1\}}\setminus \{e_1\} =\{e_1,e_2\}$, so we can remove it without changing the homotopy type of $\mathcal{U}_{2\epsilon_2}(A_2)$. Now, $\{e_1,e_2\}$ is a down beat point since $U_{\{e_1,e_2\}}\setminus \{e_1,e_2\} =\{e_2 \}$. Therefore, we can remove it. We can proceed recursively until the point $\{e_5\}$. This point satisfies that $d(e_5,e_6),d(e_5,e_{10})=\frac{1}{4}<2\epsilon_2=\frac{\sqrt{2}}{2^2}$, so $\{e_5,e_6 \}$, $\{e_5,e_{10} \}\in \mathcal{U}_{2\epsilon_2}(A_2)$, which implies that $\{ e_5\}$ is not an up beat point. On the other hand, $\{ e_5,e_6\}$ and $\{e_5,e_{10} \}$ are clearly not down beat points. A similar argument can be made with the rest of the points in $A_2$ that lie in $a_1$. In addition, the map $p_{1,2}$ trivially sends every $C\in \mathcal{U}_{2\epsilon_n}(A'_2)$ to $A_1$. In Figure \ref{fig:A_2Aproximacion} we present $A_2$ and arcs in blue to represent points of $ \mathcal{U}_{2\epsilon_n}(A_2)$ that have cardinal equal to $2$. In Figure \ref{fig:A_2AproximacionCore} we present $A_2'$. 

\underline{Step 3.} We consider $\epsilon_3=\frac{\sqrt{2}}{2^6}<\frac{\epsilon_2}{2}$, the grid $G_3=\{(\frac{l}{2^{6-1}},\frac{k}{2^{6-1}})\in \mathbb{R}^2| l,k\in \mathbb{Z}\}$ and the intersection of $G_3$ with $S$. There are $16$ points that are at distance $\epsilon_3$ to $A_3$, these points lie in $a_6$. Concretely,
$$\{(\frac{1}{2^6},\frac{2k+1}{2^6}) |k=0,1,2...,15\}.$$
We add points of $a_{\infty}$ to get an $\epsilon_3$-approximation, i.e., 
$$A_2=(G_3\cap S)\cup \{(0,\frac{2k+1}{2^6}) |k=0,1,2...,15\}. $$
We consider $B_3=\{x\in A_2|x\in a_l$ with $l=0,1,2,3$ or $x\in \overline{b}_i$ with $ i=1,2 $ or $x\in \underline{b}_1$ or $x\in \underline{b}_2\setminus a_4$  $\}$, i.e., the points of $A_3$ that lie to the right of $a_4$. We enumerate from right to left the points of $A_3$, see Figure \ref{fig:aproxtopologistssincurve3}. We have that $\{e_1\}$ is only covered by $\{e_1,e_2\}$, so we can remove it. Now, $\{e_1,e_2\}$ only covers $\{e_2\}$, so it is a down beat point and we can remove it. If we continue in this fashion, we get that $\mathcal{U}_{2\epsilon_3}(A'_3)$ is a strong deformation retract of  $\mathcal{U}_{2\epsilon_3}(A_3)$ where $A_3'=A_3\setminus B_3$. Suppose $x$ is a point of $A_3$ that lie in $a_4$. Then there exist points in $a_5$ and $a_4$ that are at distance less than $2\epsilon_3$ to $x$. Notice that every $C\in \mathcal{U}_{2\epsilon_3}(A_3')$ contains points lying in $a_4$ or $a_5$. In Figure \ref{fig:A_3AproximacionCore} we have depicted $A_3'$.

\begin{figure}[!htb]
   \begin{minipage}{0.48\textwidth}
     \centering
     \includegraphics[scale=0.45]{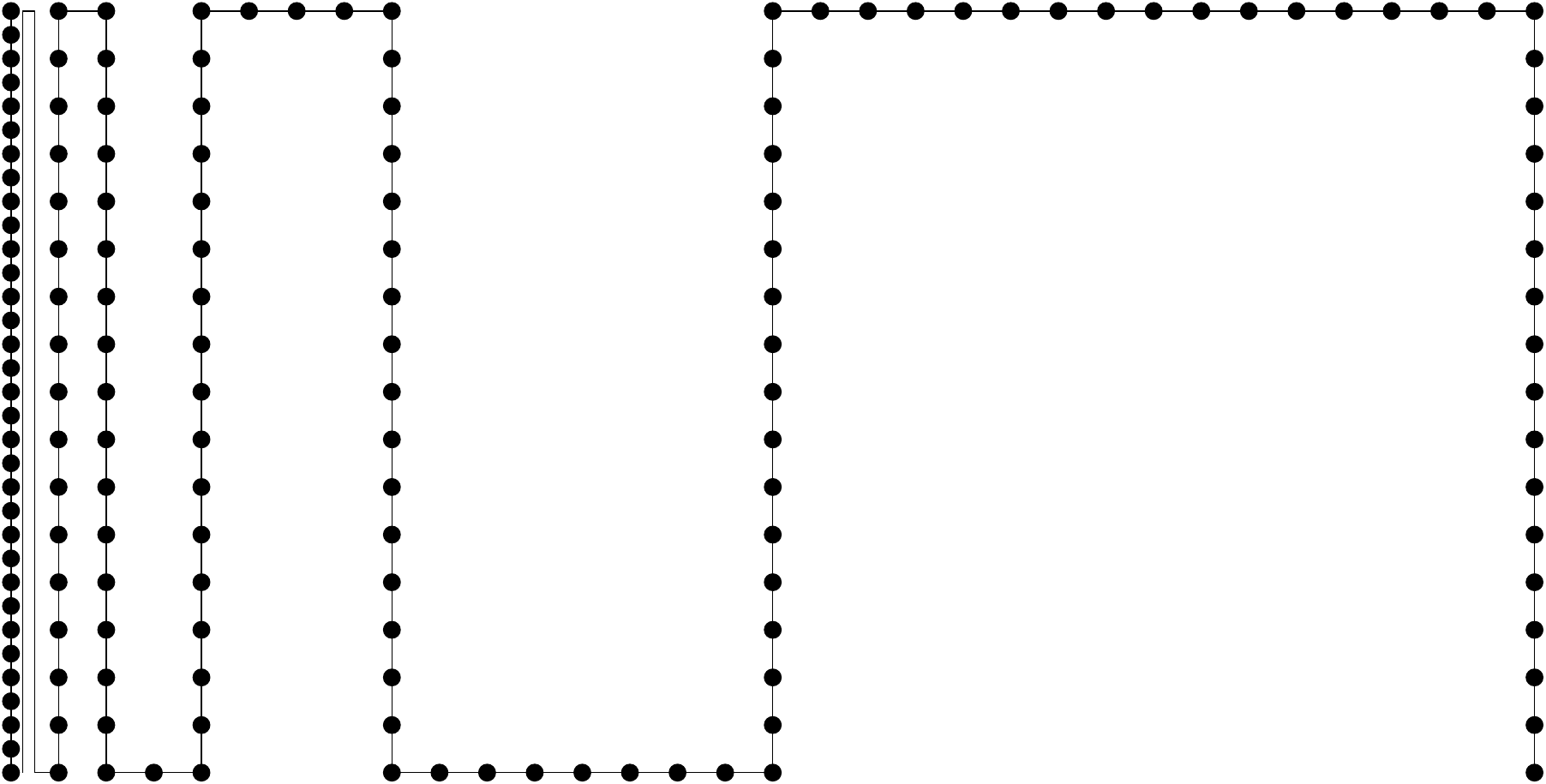}
\caption{$\epsilon_3$-approximation $A_3$ for $S$.}\label{fig:aproxtopologistssincurve3}
   \end{minipage}\hfill \quad \quad\quad
   \begin {minipage}{0.48\textwidth}
     \centering
     \includegraphics[scale=0.45]{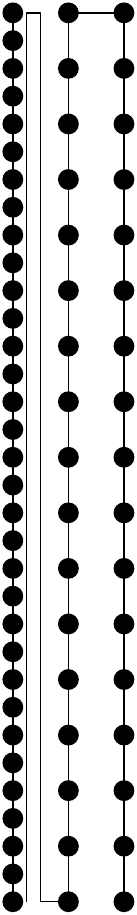}
\caption{The set of points $A_3'$ in $S$.}\label{fig:A_3AproximacionCore}
   \end{minipage}
\end{figure}
Furthermore, the image of the map $p_{2,3}:\mathcal{U}_{2\epsilon_3}(A_3')\rightarrow \mathcal{U}_{2\epsilon_2}(A_2)$ is $\mathcal{U}_{2\epsilon_2}(\overline{A}_2)\subset \mathcal{U}_{2\epsilon_2}(A_2')$, where $\overline{A}_2=\{x\in A_2|x\in a_\infty \}$. On the other hand, after some routine computations we may get that $\mathcal{U}_{2\epsilon_2}(\overline{A}_2)$ is contractible to $\{(0,\frac{1}{4})\}$. 

\underline{Step n.} We consider $\epsilon_n=\frac{\sqrt{2}}{2^{3n-3}}$, the grid $G_n=\{ (\frac{l}{2^{3n-4}},\frac{k}{2^{3n-4}})\in \mathbb{R}^2 | l,k\in \mathbb{Z}\}$ and the intersection of $G_n$ with $S$. There are $2^{3n-5}$ points that lie in $a_{3n-3}$ such that the distance to $G_n\cap S$ is exactly $\epsilon_n$. If we add the following points to $G_n\cap S$, then we get an $\epsilon_n$-approximation
$$A_n=(G_n\cap S )\cup \{(0,\frac{2k+1}{2^{3n-3}})|k=0,1...,2^{3n-4}-1 \}.$$
It is simple to show that $\mathcal{U}_{2\epsilon_n}(A_n)$ is homotopy equivalent to $\mathcal{U}_{2\epsilon_n}(A_n')$, where $A_n'=A_n\setminus B_n$ and $B_n$ consists of points in $A_n$ that lie to the right of $a_{3n-5}$. 

In addition, $p_{n-1,n}$ sends $\mathcal{U}_{2\epsilon_n}(A_n')$ to $\mathcal{U}_{2\epsilon_{n-1}}(\overline{A}_{n-1})\subset \mathcal{U}_{2\epsilon_{n-1}}(A_{n-1})$, where $\overline{A}_{n-1}=\{ x\in A_{n-1}| x\in a_{\infty}\}$. It is routine to check that $\mathcal{U}_{2\epsilon_{n-1}}(\overline{A}_{n-1})$ is contractible to $\{ (0,\frac{1}{4})\}$. 

\underline{Shape of the topologist's sine curve.} By Remark \ref{rem:mainConstruction}, $(\mathcal{U}_{2\epsilon_n}(A_n),p_{n,n+1})$ is isomorphic to every finite approximation of $S$ and $(\mathcal{U}_{2\epsilon_n}(\overline{A}_n),p_{n,n+1_{|\mathcal{U}_{2\epsilon_n}(\overline{A}_n)}})$ is isomorphic to every finite approximation of $[0,\frac{1}{2}]$ in pro-$HTop$. Thus, $S$ and $[0,\frac{1}{2}]$ are isomorphic in the shape category by Theorem \ref{thm:shapeTheorem}.
\end{ex}

We have obtained a ``computational'' description of shape theory based on finite topological spaces that are obtained from finite samples in compact metric spaces. From an algebraic point of view, this result is not surprising (see Proposition \ref{prop:reconstruccionHomologia}). Nevertheless, every finite connected space has trivial shape (see \cite{chocano2021shape}). Given a compact metric space $X$ and a finite approximation $(\mathcal{U}_{4\epsilon_n}(A_n),q_{n,n+1})$ of it, we can apply other functors. For instance, the functor $\mathcal{K}$ considered in Section \ref{sec:preliminares}. In that manner, more connections with shape theory can be found (see the notion of $HPol$-expansion in \cite{mardevsic1982shape}). 
\begin{prop}\label{prop:Hpolexpansion} Let $X$ be a compact metric space and let $(\mathcal{U}_{4\epsilon_n}(A_n),q_{n,n+1})$ be a finite approximation of $X$. Then $(\mathcal{K}(\mathcal{U}_{4\epsilon_n}(A_n)),\mathcal{K}(q_{n,n+1}))$ is a $HPol$-expansion of $X$. 
\end{prop}
This result is an immediate consequence of \cite[Theorem 12]{mondejar2015hyperspaces} and \cite[Section 5]{chocano2021computational}. It also proves that we may obtain some shape invariants applying algebraic functors (see \cite[Chapter II]{mardevsic1982shape}). Concretely, we get homology pro-groups, \v{C}ech homology, \v{C}ech cohomology and shape groups.  
\begin{ex} Combining these techniques we may deduce whether two spaces are shape equivalent. Let us consider the the Cantor set $C$ and the topological subspace of $\mathbb{R}^2$ given by the union along one edge of two squares, that is, Example 3.4 and Example 3.3 respectively in \cite{chocano2021computational}. After applying homological functors to the finite approximations of these spaces it may be observed that they are not shape equivalent (just looking at Table 1 and Table 2 in \cite{chocano2021computational}).
\end{ex}

%

To conclude this section, we define the notion of height for compact metric spaces. The height of a finite topological space $X$, denoted by $ht(X)$, is one less than the maximum number of elements in a chain of $X$. The dimension of a finite simplicial complex $L$, denoted by $\text{dim}(L)$, is the maximum of dimension of the simplices of $L$. It is clear that $\text{dim}(\mathcal{K}(X))=ht(X)$. 

\begin{df}
Let $X$ be a compact metric space. We say that a finite approximation $U=(\mathcal{U}_{4\epsilon_n}(A_n),q_{n,n+1})$ of $X$ has height $ht(U)\leq m $ if $ht(\mathcal{U}_{4\epsilon_n}(A_n))\leq m$ for every $n\in \mathbb{N}$. We say that $X$ has height $ht(X)\leq m$ provided there exists an inverse sequence of finite topological spaces $V=(V_n,t_{n,n+1})$ isomorphic in pro-$HTop$ to a finite approximation $(\mathcal{U}_{4\epsilon_n}(A_n),q_{n,n+1})$ of $X$ satisfying $ht(V)\leq m$. We write $ht(X)=n$ provided  $n$ is the least $m$ for which $ht(X)\leq m$.
\end{df}

Notice that if two compact metric spaces $X$ and $Y$ are isomorphic in $\mathbb{E}$, then $ht(X)=ht(Y)$.

The shape dimension of a compact metric space $X$, denoted by $\text{sd}(X)$, is defined similarly using the dimension of simplicial complexes and $HPol$-expansions, see \cite[Chapter II, 1]{mardevsic1982shape}. The following result is an immediate consequence of the definitions.

\begin{prop} Let $X$ be a compact metric space. Then $\textnormal{sd}(X)\leq ht(X)$.
\end{prop}

\begin{ex} Let us consider the topologist's sine curve $S$ and the inverse sequences of finite posets considered in Example \ref{ex:curve}. Then $ht(X)=0$, which coincides with the shape dimension of $S$. 
\end{ex}


\section{Proof of Theorem \ref{thm:shapeTheorem}}\label{sec:resultadoPrincipal}

Given two compact metric spaces $X$ and $Y$, we choose finite approximations for them $T(X)=(\mathcal{U}_{4\delta_n}(A_n),q_{n,n+1})$ and $T(Y)=(\mathcal{U}_{4\epsilon_n}(B_n),q_{n,n+1})$. We prove that the set of morphisms in pro-$HTop$ between their finite approximations is in bijective correspondence with the set of shape morphisms between $X$ and $Y$. To this end, we first show some technical results.
\begin{lem}\label{lem:definiciondep:x} Let $(X,d)$  be a compact metric spaces and let $A$ be an $\epsilon$-approximation of $X$. Then the map $p:X\rightarrow \mathcal{U}_{4\epsilon}(A)$ given by $p(x)=\{a\in A|d(x,a)=d(x,A) \}$ is well-defined and continuous.
\end{lem}
\begin{proof}
The proof follows easily from \cite[Lemma 2]{mondejar2015hyperspaces}.
\end{proof}
\begin{lem}\label{lem:unionhomotopica} Let $X$ and $Y$ be a topological space and a compact metric space respectively. Suppose $A$ is a finite subset of $Y$. If $f,g:X\rightarrow U_{\epsilon}(Y)$ ($f,g:X\rightarrow \mathcal{U}_\epsilon(A)$) are continuous maps where $\epsilon$ is a positive real value and $f\cup g:X\rightarrow U_{\epsilon}(Y)$ ($f\cup g:X\rightarrow \mathcal{U}_\epsilon(A)$) given by $f\cup g(x)=f(x)\cup g(x)$ is well-defined, then $f\cup g$ is continuous and homotopic to $f$ and $g$.
\end{lem}
\begin{proof}
For simplicity we denote the map $f\cup g$ by $h$. We prove the continuity of $h$. Let $x\in X$. If $U$ is an open set containing $h(x)$, then it also contains $f(x)$ and $g(x)$. From the continuity of $f$ and $g$, it follows that there exist open sets $V_f$ and $V_g$ containing $x$ such that $f(V_f)\subseteq U$ and $f(V_g)\subseteq U$. Hence, we get the continuity of $h$.

We show that $h$ is homotopic to $f$. We consider $H:X\times [0,1]\rightarrow Z$ given by $H(x,t)=f(x)$ if $t\in[0,1)$ and $H(x,1)=h(x)$, where $Z$ is $U_{\epsilon}(Y)$ or $\mathcal{U}_\epsilon(A)$. It suffices to check the continuity of $H$ at points of the form $(x,1)$ where $x\in X$. Let $(x,1)\in X\times \{1\}$. Since $h$ is continuous, for every open set $U$ containing $h(x)$ there exists an open set $V$ containing $x$ such that $f(V)\cup g(V)\subseteq U$. Concretely, $f(V)\subseteq U$. Hence, $V\times I$ is an open set of $X\times [0,1]$ containing $(x,1)$ and satisfying that $H(V\times [0,1])\subseteq U$, which gives the desired result. The proof to show that $g$ is homotopic to $h$ is the same.
\end{proof}

In the following proposition we get a constructive method to get a morphism in pro-$HTop$ induced by the homotopy class of an approximative map.
\begin{prop}\label{prop:approxFASmap} Let $(X,d)$ and $(Y,l)$ be compact metric spaces. If $[\overline{f}]:X\rightarrow Y$ is the homotopy class of an approximative map, then there exists a natural morphism $T([\overline{f}]):T(X)\rightarrow T(Y)$ in pro-$HTop$.
\end{prop}
\begin{proof}
Set $T(X)=(\mathcal{U}_{4\delta_n}(A_n),q_{n,n+1})$ and $T(Y)=(\mathcal{U}_{4\epsilon_n}(B_n),q_{n,n+1})$. Firstly, we prove that for every $n\in \mathbb{N}$ the map $r_n:U_{2\epsilon_n}(Y)\rightarrow \mathcal{U}_{4\epsilon_n}(B_n)$ given by 
$$r_n(C)=\bigcup_{x\in C}\{b\in B_n|l(x,b)=l(x,B_n) \}$$
is well-defined and continuous. If $x,y\in r_n(C)$ for some $C\in \mathcal{U}_{2\epsilon_n}(Y)$, then there exist $c_x,c_y\in C$ such that $x\in r_n(c_x)$ and $y\in r_{n}(c_y)$. We have $l(x,c_x),l(y,c_y)<\epsilon_n$ and $l(c_x,c_y)<2\epsilon_n$. Hence, we get 
$$l(x,y)<l(x,c_x)+l(c_x,c_y)+l(c_y,y)<\epsilon_n+2\epsilon_n+\epsilon_n, $$
which implies that $r_n$ is well-defined. We get that $r_n$ is continuous because $r_n=p_{n_{|\mathcal{U}_{2\epsilon_n}(Y)}}^*$, where $p_n$ is the continuous map considered in Lemma \ref{lem:definiciondep:x} and $p_n^*$ denotes the extension of $p_n$ to the hyperspace of $Y$ defined in Lemma \ref{lem:elevacionHiperespacio}. 

We now prove the commutativity up to homotopy of the following diagram, where $i$ denotes the inclusion map.
\[
  \begin{tikzcd}[row sep=large,column sep=huge]
  U_{2\epsilon_n}(Y) \arrow{r}{r_n} &  \mathcal{U}_{4\epsilon_n}(B_n)  \\  U_{2\epsilon_{n+1}}(Y) \arrow{u}{i} \arrow{r}{r_{n+1}}& \mathcal{U}_{4\epsilon_{n+1}}(B_{n+1}) \arrow{u}{q_{n,n+1}}
  \end{tikzcd}
\]
Consider $C\in \mathcal{U}_{2\epsilon_{n+1}}(Y)$. If $x\in q_{n,n+1}(r_{n+1}(C))$, then there exist $a_x\in B_{n+1}$ and $b_x\in C$ such that $x\in q_{n,n+1}(a_x)$ and $a_x\in r_{n+1}(b_x)$. We obtain
$$l(x,b_x)<l(x,a_x)+l(a_x,b_x)<\epsilon_n+\frac{\epsilon_n}{2}.$$
If $y\in r_n(i(C))$, then there exists $b_y\in C$ such that $y\in r_n(b_y)$. Since $l(b_x,b_y)<2\epsilon_{n+1}<\epsilon_n$, we have
$$l(x,y)<l(x,b_x)+l(b_x,b_y)+l(b_y,y)<\epsilon_n+\frac{\epsilon_n}{2}+\epsilon_n+\epsilon_n.$$

Thus, we have $h=r_n\circ i \cup q_{n,n+1}\circ r_{n+1}:U_{2\epsilon_{n+1}}(Y)\rightarrow \mathcal{U}_{4\epsilon_n(B_n)}$ is well-defined, and so the continuity of $h$ and the commutativity up to homotopy of the previous diagram follow from Lemma \ref{lem:unionhomotopica}.

Let us consider $f=\{f_k:X\rightarrow 2^Y\}_{k\in \mathbb{N}}\in [\overline{f}]$. By Proposition \ref{prop:baseEntornosAbiertos}, for every open neighborhood $U$ of $Y$ in $2^Y$ there exists $n$ such that $U_{2\epsilon_n}(Y)\subseteq U$. By Definition \ref{def:approximativeMap}, there exists $s(n)$ such that $f_m$ is homotopic to $f_{m+1}$ in $U_{2\epsilon_n}(Y)$ for every $m\geq s(n)$. Let $H$ denote the homotopy between $f_{s(n)}$ and $f_{s(n)+1}$. By Theorem \ref{thm:extensionHomotopiaHyperspaces}, there exists $\gamma_n>0$ such that $\overline{H}(U_{\gamma_n}(X)\times I)\subseteq U_{2\epsilon_n}(Y)$. We also denote by $f$ the map $f:\mathbb{N}\rightarrow \mathbb{N}$ given by $f(n)=\min\{ l\in \mathbb{N}|4\delta_l<\gamma_n\}$, it is clear that $f$ is well-defined and satisfies that $f(n)\leq f(m)$ for every $n\leq m$ in $\mathbb{N}$. For every natural number $n$ we consider $f_{n}:\mathcal{U}_{4\delta_{f(n)}}(A_{f(n)})\rightarrow \mathcal{U}_{4\epsilon_n}(B_n)$ given by $f_n=r_n\circ f_{s(n)}^*\circ i$, where $f_{s(n)}^*$ denotes the extension of $f_{s(n)}$ to the hyperspace of $X$ given in Lemma \ref{lem:elevacionHiperespacio} and $i:\mathcal{U}_{4\delta_{f(n)}}(B_n)\rightarrow U_{\gamma_n}(X)$ denotes the inclusion map. By construction, it is immediate to get that $f_n$ is well-defined and continuous for every $n\in \mathbb{N}$. To check that $(f_n,f):T(x)\rightarrow T(Y)$ is a morphism in pro-$HTop$ we need to verify the commutativity up to homotopy of the following diagram.
\[
  \begin{tikzcd}[row sep=large,column sep=huge]
\mathcal{U}_{4\delta_{f(n)}}(A_{f(n)}) \arrow{r}{i} & U_{\gamma_{n}}(X) \arrow{r}{f^*_{s(n)}} & U_{2\epsilon_n}(Y) \arrow{r}{r_n} & \mathcal{U}_{4\epsilon_n}(B_n) \\
\mathcal{U}_{4\delta_{f(n+1)}}(A_{f(n+1)}) \arrow{r}{i} \arrow{u}{q}& U_{\gamma_{n+1}}(X) \arrow{u}{i} \arrow{r}{f^*_{s(n+1)}} & U_{2\epsilon_{n+1}}(Y)\arrow{r}{r_{n+1}} \arrow{u}{i} & \mathcal{U}_{4\epsilon_{n+1}}(B_{n+1})\arrow{u}{q}
  \end{tikzcd}
\]
We check the commutativity up to homotopy of the first square. Consider $C\in \mathcal{U}_{4\delta_{f(n+1)}}(A_{f(n+1)})$. If $x\in i(q(C))=q(C)$, then there exists $a_x\in C$ such that $x\in q(a_x)$, which implies $d(x,a_x)<2\delta_{f(n)}<\frac{\gamma_n}{2}$. If $y\in i(i(C))=C$, then $d(y,a_x)<4\delta_{f(n+1)}<\frac{\gamma_n}{2}$. We get $d(x,y)<\gamma_n$. Therefore, $q(C)\subseteq q(C)\cup C$ for every $C\in \mathcal{U}_{4\delta_{f(n+1)}}(A_{f(n+1)})$, where $q\cup i:\mathcal{U}_{4\delta_{f(n+1)}}(A_{f(n+1)})\rightarrow U_{\gamma_n}(X)$ is well-defined and continuous. By Lemma \ref{lem:unionhomotopica}, the commutativity up to homotopy of the first square can be deduced. The second square is commutative up to homotopy by construction. The commutativity up to homotopy of the third square was proved at beginning. 

If $g\in[\overline{f}]$, then we can repeat the same construction to get $(g_n,g):T(X)\rightarrow T(Y)$. We prove that $(f_n,f)$ is equivalent to $(g_n,g)$ as morphisms in pro-$HTop$. To do this, given a natural number $n$ we need to verify the commutativity up to homotopy of the following diagram for some $m\geq f(n),g(n)$.
\[
  \begin{tikzcd}
& \mathcal{U}_{4\delta_{f(n)}}(A_{f(n)}) \arrow{r}{i} & U_{\gamma_{n}}(X)\arrow{dd}{i} \arrow{r}{f^*_{s(n)}} & U_{2\epsilon_n}(Y)\arrow{dd}{\textnormal{id}} \arrow{dr}{r_n} & \\
\mathcal{U}_{4\delta_m}(A_m)\arrow{ur}{q}\arrow{dr}{q} & & & &  \mathcal{U}_{4\epsilon_n}(B_n) \\
& \mathcal{U}_{4\delta_{g(n)}}(A_{g(n)}) \arrow{r}{i} & U_{\tau_{n}}(X) \arrow{r}{g^*_{h(n)}} & U_{2\epsilon_n}(Y)  \arrow{ur}{r_n}
  \end{tikzcd}
\]
We define $m=\max\{f(n),g(n) \}$. Without loss of generality we can assume that $\gamma_n\leq \tau_n$ and $m=f(n)$. We study the commutativity up to homotopy of the first square. Suppose $f(n)\neq g(n)$ because the other case follows easily. Consider $C\in \mathcal{U}_{4\delta_m}(A_m)$. If $y\in i(q(C))=q(C)$, then there exists $a_y\in C$ such that $y\in q(a_y)$. We get $d(a_y,y)<2\delta_{g(n)}<\frac{\tau_n}{2}$ and $d(x,a_y)<4\delta_{f(n)}<2\delta_{g(n)}<\frac{\tau_n}{2}$ for every $x\in C$. Thus, we have 
$$d(x,y)<d(a_x,a_y)+d(a_y,y)<\tau_n, $$
and consequently $i\circ q \cup i\circ i :\mathcal{U}_{4\delta_{f(n)}}(A_{f(n)})\rightarrow U_{\tau_n}(X)$ is well-defined. By Lemma \ref{lem:unionhomotopica}, we get the desired result. The commutativity up to homotopy of the second square follows from the fact that $f$ and $g$ are homotopic approximative maps (see Theorem \ref{thm:extensionHomotopiaHyperspaces}) and the choice of $s(n)$ and $h(n)$. For every $m\geq s(n),h(n)$ we have that $f_{s(n)}$ is homotopic to $f_m$ in $U_{2\epsilon_n}(Y)$ and $g_{h(n)}$ is homotopic to $g_m$ in $U_{2\epsilon_n}(Y)$. The third square commutes trivially. 
\end{proof}

Given a morphism in pro-$HTop$ we construct a homotopy class of an approximative map.
\begin{prop}\label{prop:FAStoApprox} Let $(X,d)$ and $(Y,l)$ be compact metric spaces. If $(f_n,f):T(X)\rightarrow T(Y)$ is a morphism in pro-$HTop$, then there exists a natural homotopy class of an approximative map $E(f_n):X\rightarrow Y$ induced by $(f_n,f)$.
\end{prop}
\begin{proof}
Set $T(X)=(\mathcal{U}_{4\delta_n}(A_n),q_{n,n+1})$ and $T(Y)=(\mathcal{U}_{4\epsilon_n}(B_n),q_{n,n+1})$. Firstly, we consider $p_n:X\rightarrow \mathcal{U}_{4\delta_n}(A_n)$ given by $p_n(x)=\{a\in A_n| d(x,a)=d(x,A_n)\}$ for every $n\in \mathbb{N}$, that is, the continuous map considered in Lemma \ref{lem:definiciondep:x}. We prove that the following diagram commutes up to homotopy.
\[
  \begin{tikzcd}[row sep=large,column sep=huge]
& X \arrow[dl,"p_n"'] \arrow{dr}{p_{n+1}} & \\
\mathcal{U}_{4\delta_n}(A_n) & & \mathcal{U}_{4\delta_{n+1}(A_{n+1})} \arrow{ll}{q_{n,n+1}}
  \end{tikzcd}
\]
If $a_x\in q(p(x))$, then there exists $b_x\in A_{n+1}$ with $a_x\in q(b_x)$ and $b_x\in p(x)$. Hence, $d(x,b_x)<\delta_{n+1}<\frac{\delta_n}{2}$ and $d(a_x,b_x)<2\delta_n$. If $c_x\in p(x)$, then we get $d(x,c_x)<\delta_n$. Thus,
$$d(c_x,a_x)<d(c_x,x)+d(x,b_x)+d(b_x,a_x)<3\delta_n ,$$
which implies that $q_{n,n+1}\circ p_{n+1}\cup p_n:X\rightarrow \mathcal{U}_{4\delta_n}(A_n)$ is well-defined. Applying Lemma \ref{lem:unionhomotopica} we get that the previous diagram commutes up to homotopy.

We construct a candidate to be an approximative map. We consider $F=\{F_k:X\rightarrow 2^Y\}_{k\in \mathbb{N}}$ given by $F_k=f_k\circ p_{f(k)}:X\rightarrow \mathcal{U}_{4\epsilon_k}(B_k)$. For every open neighborhood $U$ of the canonical copy of $Y$ in $2^Y$, there exists $4\epsilon_m$ such that $\mathcal{U}_{4\epsilon_m}(B_m)\subset U_{4\epsilon_m}(Y)\subseteq U$ by Proposition  \ref{prop:baseEntornosAbiertos}. To prove that $F$ is an approximative map we check that $F_{m}$ is homotopic to $F_{m+1}$ in $U$, which is equivalent to show the commutativity up to homotopy of the following diagram.
\[
  \begin{tikzcd}[row sep=large,column sep=huge]
X \arrow{r}{p_{f(m)}}& \mathcal{U}_{4\delta_{f(m)}}(A_{f(m)}) \arrow{r}{f_{m}} & \mathcal{U}_{4\epsilon_m} (B_m) \\ 
X \arrow{u}{\textnormal{id}} \arrow{r}{p_{f(m+1)}} & \mathcal{U}_{4\delta_{f(m+1)}}(A_{f(m+1)}) \arrow{u}{q} \arrow{r}{f_{m+1}} & \mathcal{U}_{4\epsilon_{m+1}}(B_{m+1}) \arrow{u}{q}
  \end{tikzcd}
\]
The commutativity up to homotopy of the first square was proved at the beginning.
The second square commutes up to homotopy since $(f_n,f)$ is a morphism in pro-$HTop$. Thus $F$ is an approximative map. We denote by $E(f_n)$ the homotopy class generated by the approximative map $F$. We verify that $E$ is well-defined, that is, if $(g_n,g)$ is equivalent to $(f_n,f)$ as morphisms in pro-$HTop$, then the induced approximative map $G=\{G_k=g_k\circ p_{f(k)}:X\rightarrow 2^Y\}_{k\in \mathbb{N}}$ is homotopic to $F=\{F_k=f_k \circ p_{f(k)}:X\rightarrow 2^Y\}_{k\in \mathbb{N}}$. 

For every open neighborhood $U$ of $Y$ in $2^Y$ there exists $4\epsilon_n$ such that $U_{4\epsilon_n}(Y)\subseteq U$ by Proposition \ref{prop:baseEntornosAbiertos}. By hypothesis, for every $n$ there exists $m\geq f(n),g(n)$ such that $f_{n}\circ q_{f(n),m}$ is homotopic to $g_{n} \circ q_{g(n),m}$. We have the following diagram.

\[
  \begin{tikzcd}[row sep=large,column sep=huge]
  X      \arrow[d,"p_{f(n)}"]           &              X \arrow[l,"\text{id}"'] \arrow{r}{\text{id}}   \arrow[d,"p_{m}"]            &                      X \arrow[d,"p_{g(n)}"]   \\
  \mathcal{U}_{4\delta_{f(n)}}(A_{f(n)}) \arrow[dr, "f_n"'] &  \mathcal{U}_{4\delta_{m}}(A_{m}) \arrow[l, "q_{f(n),m}"'] \arrow[r, "q_{g(n),m}"]  & \mathcal{U}_{4\delta_{g(n)}}(A_{g(n)}) \arrow[dl, "g_n"]\\
  & \mathcal{U}_{4\epsilon_n}(B_n) &
    \end{tikzcd} 
\]

It is obvious that every square commutes up to homotopy. From here, we obtain that $F$ and $G$ are homotopic.
\end{proof}

\begin{lem}\label{lem:ETidentidad} Let $(X,d)$ and $(Y,l)$ be compact metric spaces. If $[\overline{f}]:X\rightarrow Y$ is the homotopy class of an approximative map, then $E(T([\overline{f}]))=[\overline{f}]$.
\end{lem}
\begin{proof}
Let $\overline{f}=\{\overline{f}_k: X\rightarrow 2^Y \}_{k\in \mathbb{N}}$ denote the approximative map that generates $[\overline{f}]$. We consider a representative $(f_n,f)$ of $T([\overline{f}])$ induced by $\overline{f}$ and a representative $\overline{f}'=\{\overline{f}_k':X\rightarrow 2^Y\}_{k\in \mathbb{N}}$ of $E(T(f_n))$ induced by $(f_n,f)$. Therefore, $\overline{f}_k'=f_k\circ p_{f(k)}$, where  $f_k=r_{k}\circ \overline{f}^*_{s(k)}\circ i$, see the proof of Proposition \ref{prop:approxFASmap}.
For every open neigbhorhood $U$ of $Y$ in $2^Y$ there exists $U_{4\epsilon_n}(Y)\subset U$ by Proposition \ref{prop:baseEntornosAbiertos}. We check that the following diagram is commutative up to homotopy. 
\[
  \begin{tikzcd}[row sep=large,column sep=huge]
  X \arrow{d}{\overline{f}_{s(n)}} \arrow{r}{p_{f(n)}} & \mathcal{U}_{4\delta_{f(n)}}(A_{f(n)}) \arrow{r}{i}   & \mathcal{U}_{\gamma_n}(X) \arrow{dl}{\overline{f}_{s(n)}^*} \\ \mathcal{U}_{2\epsilon_n}(Y) \arrow{d}{i} &  \mathcal{U}_{2\epsilon_n}(Y) \arrow{dl}{r_n} \arrow{l}{\text{id}} \\
  \mathcal{U}_{4\epsilon_n}(Y)
    \end{tikzcd} 
\]
Let $x\in X$. Then we know that $\text{diam}(\overline{f}_{s(n)}(x))<2\epsilon_n$. We have $d(x,a)<\delta_{f(n)}<\gamma_n$ for every $a\in p(x)$. Thus, $2\epsilon_n> \text{diam}(\overline{f}^*_{s(n)}(i(p(x))\cup \{x\}))=\text{diam}(\overline{f}^*_{s(n)}(i(p(x)))\cup \overline{f}_{s(n)}(x))$ so $h=\overline{f}_{s(n)}\cup \overline{f}_{s(n)}^* \circ i \circ p_{f(n)}:X\rightarrow U_{\epsilon_n}(Y)$ is continuous and well-defined.  By Lemma \ref{lem:unionhomotopica}, we get the commutativity up to homotopy of the first square. It is clear that $\text{diam}(i(C)\cup r_n(C))<4\epsilon_n$ for every $C\in U_{2\epsilon_n}(Y)$. Therefore, we can repeat the previous argument to show that $r_n$ is homotopic to $i$. 

By construction and hypothesis, for every $m\geq s(n),n$ we have that $\overline{f}_m$ is homotopic to $\overline{f}_{s(n)}$ in $U$ and $\overline{f}_{n}'$ is homotopic to $\overline{f}_m'$ in $U$, which gives the desired result.
\end{proof}

\begin{lem}\label{lem:TEidentidad} Let $(X,d)$ and $(Y,l)$ be compact metric spaces. If $(f_n,f):T(X)\rightarrow T(Y)$ is a morphism in pro-$HTop$, then $T(E(f_n,f))=(f_n,f)$.
\end{lem}
\begin{proof}
We consider a representative $(f_n',f')$ of $E(T(f_n))$, where $(f_n',f')$ is induced by an approximative map $F$ induced by $(f_n,f)$. Given $n\in \mathbb{N}$, we need to verify that for every $l\geq f'(n),f(n)$ the following diagram commutes up to homotopy.
\[
  \begin{tikzcd}[row sep=large,column sep=huge]
  \mathcal{U}_{4\delta_{f(n)}}(A_{f(n)}) \arrow[dr,"f_n"']  & \mathcal{U}_{4\delta_{l}}(A_{l})  \arrow[l,"q_{f(n),l}"'] \arrow{r}{q_{f'(n),l}}& \mathcal{U}_{4\delta_{f'(n)}}(A_{f'(n)})\arrow[dl,"f'_n"] \\
  &\mathcal{U}_{4\epsilon_n(B_n)} & 
    \end{tikzcd} 
\] 

Without loss of generality we can assume that $f'(n)>f(n)$. Then, we need to check the commutativity up to homotopy of the following diagram, where $h=(f_{s(n)}\circ p_{f(s(n))})^*$.
\[
  \begin{tikzcd}[row sep=large,column sep=huge]
\mathcal{U}_{4\delta_{f'(n)}}(A_{f'(n)}) \arrow{d}{q} \arrow{r}{i} & \mathcal{U}_{\gamma_n}(X)\arrow{r}{h} & \mathcal{U}_{2\epsilon_n}(Y)\arrow{r}{r_n} & \mathcal{U}_{4\epsilon_n} (B_n) \\
\mathcal{U}_{4\delta_{f(n)}}(B_{f(n)}) \arrow{urrr}{f_n} & & & 
    \end{tikzcd} 
\] 
By construction, $h(C)=\bigcup_{c\in C} f_{s(n)}(p_{f(s(n))}(c))$ for every $C\in \mathcal{U}_{4\delta_{f'(n)}}(A_{f'(n)})$. We have $f'(n)\geq f(s(n))$ and $s(n)\geq n$. Then, it is easy to show that $q_{f'(n),f(s(n))}$ is homotopic to $p^*_{f(s(n))}\circ i$ because $p_{f(s(n))}(i(c))\subseteq q_{f'(n),f(s(n))}(c)$ for every $c\in C$. In addition, $r_n$ restricted to the image of $h\circ i$ is homotopic to $q_{s(n),n}$ because $r_n(a)\subseteq q_{s(n),n}(a)$. Since $(f_n,f)$ is a morphism in pro-$HTop$, it follows the commutativity up to homotopy of the diagram.
\end{proof}


\begin{thm}\label{thm:ShapeDescriptionNew} Let $(X,d)$ and $(Y,l)$ be compact metric spaces. The set of shape morphisms between $X$ and $Y$ is in bijective correspondence with the set of morphisms in pro-$HTop$ between $T(X)$ and $T(Y)$.
\end{thm}
\begin{proof}
We consider the constructions made in Proposition \ref{prop:approxFASmap} and Proposition \ref{prop:FAStoApprox}. Thus, the result is an immediate consequence of Lemma \ref{lem:ETidentidad} and Lemma \ref{lem:TEidentidad}.
\end{proof}

The task is now to prove that $T$ is indeed a functor. 
\begin{lem} Let $[\overline{f}]:X\rightarrow Y$ and $[\overline{g}]:Y\rightarrow Z$ be two approximative maps. Then $T([\overline{g}]\circ [\overline{f}])=T([\overline{g}])\circ T([\overline{f}])$. 
\end{lem}
\begin{proof}
We consider $\{f_n:X\rightarrow 2^Y\}_{n\in \mathbb{N}}\in [\overline{f}]$, $\{g_n:X\rightarrow 2^Y\}_{n\in \mathbb{N}}\in [\overline{g}]$ and $\{g_n\circ f_{k_n}:X\rightarrow 2^Z\}_{n\in \mathbb{N}}\in [\overline{g}]\circ [\overline{f}]$. From the proof of Theorem \ref{prop:approxFASmap} we get that $(r_n\circ f^*_{s(n)}\circ i,f)$ is a morphism in pro-$HTop$ induced by $\{f_n:X\rightarrow 2^Y\}_{n\in \mathbb{N}}$, $(r_n\circ g^*_{h(n)}\circ i,g)$ is a morphism in pro-$HTop$ induced by $\{g_n:X\rightarrow 2^Y\}_{n\in \mathbb{N}}$ and $(r_n\circ h^*_{l(n)}\circ i,h)$ is a morphism in pro-$HTop$ induced by $\{g_n\circ f_{k_n}:X\rightarrow 2^Z\}$, where $h_n=g^*_n\circ f^*_{k_n}$, $T(X)=(\mathcal{U}_{4\delta_n}(A_n),q_{n,n+1})$, $T(Y)=(\mathcal{U}_{4\epsilon_n}(B_n),q_{n,n+1})$ and $T(Z)=(\mathcal{U}_{4\psi_n}(C_n),q_{n,n+1})$. It suffices to show that $(r_n\circ g^*_{h(n)}\circ i,g)\circ (r_n\circ f^*_{s(n)}\circ i,f)$ is homotopic to $(r_n\circ h^*_{l(n)}\circ i,h)$, that is, the following diagram commutes up to homotopy for some $m\geq h(n),g(f(n))$.

\[
  \begin{tikzcd}[row sep=normal,column sep=normal]
\mathcal{U}_{4\delta_{f(h(n))}}(A_{f(h(n))}) \arrow{r}{i} & \mathcal{U}_{\gamma_n}(X)\arrow{r}{f^*_{s(h(n))}} & \mathcal{U}_{2\epsilon_{h(n)}}(Y)\arrow{r}{r_{h(n)}} & \mathcal{U}_{4\epsilon_{h(n)}} (B_{h(n)}) \arrow{r}{i} & \mathcal{U}_{\tau}(Y) \arrow{r}{g^*_{h(n)}} & \mathcal{U}_{2\psi_n}(Z) \arrow{d}{r_n} &  \\
\mathcal{U}_{4\delta_m}(A_m) \arrow{u}{q} \arrow{d}{q} & & & & &  \mathcal{U}_{4\psi_n}(C_n) & \\
\mathcal{U}_{4\delta_{l(n)}}(A_{l(n)}) \arrow{r}{i} &  \mathcal{U}_{\rho_n}(X) \arrow{rrrr}{h_{l(n)}^*} & &  & & \mathcal{U}_{2\psi_n}(Z) \arrow{u}{r_n}  &
    \end{tikzcd} 
\] 
Without loss of generality we can assume that $m=max\{l(n),f(g(n)) \}=l(n)$ and $\gamma_n\geq \rho_n$. Hence, the commutativity up to homotopy of the following diagram follows trivially.
\[
  \begin{tikzcd}[row sep=normal,column sep=huge]
\mathcal{U}_{4\delta_{f(h(n))}}(A_{f(h(n))}) \arrow{r}{i} & \mathcal{U}_{\gamma_n}(X) \\
\mathcal{U}_{4\delta_{l(n)}}(A_{l(n)}) \arrow{u}{q}  \arrow{r}{i} &  \mathcal{U}_{\rho_n}(X) \arrow{u}{i}
    \end{tikzcd} 
\] 
We verify the commutative up to homotopy of the following diagram.

\[
  \begin{tikzcd}[row sep=normal,column sep=huge]
  \mathcal{U}_{2\epsilon_{h(n)}(Y)}\arrow{r}{r_{h(n)}} \arrow{dr}{g^*_{h(n)}} & \mathcal{U}_{4\epsilon_{h(n)}}(B_{h(n)})\arrow{d}{g^*_{h(n)}} \\
   & \mathcal{U}_{2\psi_n}(Z) 
    \end{tikzcd} 
\] 
For every $C\in \mathcal{U}_{2\epsilon_{h(n)}(Y)}$ we get that $\text{diam}(C\cup r_{h(n)}(C))<4\epsilon_{h(n)}$. Hence, $g^*_{h(n)}\cup g^*_{h(n)} \circ r_{h(n)}:\mathcal{U}_{2\epsilon_{h(n)}(Y)}\rightarrow \mathcal{U}_{2\psi_n}(Z) $ given by $g^*_{h(n)}\cup g^*_{h(n)} \circ r_{h(n)}(C)=g^*_{h(n)}(C)\cup g^*_{h(n)}(r_{h(n)}(C))$ is well-defined. By Lemma \ref{lem:unionhomotopica} we get that the above diagram is commutative up to homotopy.

From the commutativity of the previous diagrams and the properties of approximative maps, it follows that the first diagram commutes up to homotopy.
\end{proof}

\begin{lem} Let $[ \textnormal{id} ]$ be the class of the identity morphism $\{ \textnormal{id}_n:X\rightarrow 2^X\}$. Then $T([\textnormal{id}]):T(X)\rightarrow T(X)$ is homotopic to the identity morphism in pro-$HTop$. 
\end{lem}
\begin{proof}
We have $T(X)=(\mathcal{U}_{4\delta_n}(A_n),q_{n,n+1})$. Consider $(I_n,I)\in T([\text{id}])$ given by Proposition \ref{prop:approxFASmap}, that is, $I_n=r_n$ and $I:\mathbb{N}\rightarrow\mathbb{N}$ is the identity map. By the definition of $r_n $ it is easily seen that $I_n$ is the identity map. Therefore, it can be deduced that $T([\text{id}])$ is the identity morphism in pro-$HTop$.
\end{proof}

Combining previous results the proof of Theorem \ref{thm:shapeTheorem} is straightforward.

\bibliography{bibliografia}
\bibliographystyle{plain}

\newcommand{\Addresses}{{
  \bigskip
  \footnotesize

  \textsc{ P.J. Chocano, Departamento de Matemática Aplicada,
Ciencia e Ingeniería de los Materiales y
Tecnología Electrónica, ESCET
Universidad Rey Juan Carlos, 28933
Móstoles (Madrid), Spain}\par\nopagebreak
  \textit{E-mail address}:\texttt{pedro.chocano@urjc.es}
  
  \medskip

\textsc{ M.A. Mor\'on,  Departamento de \'Algebra, Geometr\'ia y Topolog\'ia, Universidad Complutense de Madrid and Instituto de
Mat\'ematica Interdisciplinar, Plaza de Ciencias 3, 28040 Madrid, Spain}\par\nopagebreak
  \textit{E-mail address}: \texttt{ma\_moron@mat.ucm.es}

  \medskip

\textsc{ F. R. Ruiz del Portal,  Departamento de \'Algebra, Geometr\'ia y Topolog\'ia, Universidad Complutense de Madrid and Instituto de
Matem\'atica Interdisciplinar
, Plaza de Ciencias 3, 28040 Madrid, Spain}\par\nopagebreak
  \textit{E-mail address}: \texttt{R\_Portal@mat.ucm.es}

}}

\Addresses

\end{document}